\documentclass[final]{siamltex}

\usepackage[leqno]{amsmath}
\usepackage{amsmath,amssymb,amsfonts}
\usepackage{graphicx}
\usepackage{manfnt}
\usepackage[all]{xy}
\usepackage{enumerate}
\usepackage{mathrsfs}
\usepackage{color}

\title{A small-gain theorem for nonlinear stochastic systems with inputs and outputs I: Additive white noise\thanks{This work was supported by the National Natural Science Foundation of China (NSFC) under
Grants No.11371252 and No.11501369; Research and Innovation Project of Shanghai Education Committee under Grant No.14zz120; Yangfan Program of Shanghai (14YF1409100); Chen Guang Project(14CG43) of Shanghai Municipal Education Commission, Shanghai Education Development Foundation; the Research Program of Shanghai Normal University (SK201403) and Shanghai Gaofeng Project for University Academic Program Development.}}

\author{Jifa Jiang\thanks{Department of Mathematics, Shanghai Normal
University, Shanghai 200234, PR China ({\tt jiangjf@shnu.edu.cn}).}
        \and Xiang Lv\thanks{Corresponding author. Department of Mathematics, Shanghai Normal
University, Shanghai 200234, PR China ({\tt lvxiang@shnu.edu.cn}).}}

\begin{document}

\maketitle

\begin{abstract}
This paper studies a small-gain theorem for nonlinear stochastic equations driven by additive white noise in both trajectories and stationary distribution. Motivated by the most recent work of Marcondes de Freitas and Sontag \cite{FS3}, we firstly define the {\it ``input-to-state characteristic operator"} $\mathcal{K}(u)$ of the system in a suitably chosen input space via backward It\^{o} integral, and then for a given output function $h$, define the $``gain\ operator"$  as the composition of output function $h$ and the input-to-state characteristic operator $\mathcal{K}(u)$ on the input space. Suppose that the output function is either order-preserving or anti-order-preserving in the usual vector order and the global Lipschitz constant of the output function is less than the absolute of the negative principal eigenvalue of linear matrix. Then we prove the so-called {\it ``small-gain theorem"}: the gain operator has a unique fixed point, the image for input-to-state characteristic operator at the fixed point is a globally attracting stochastic equilibrium for the random dynamical system generated by the stochastic system. Under the same assumption for the relation between the Lipschitz constant of the output function and maximal real part of stable linear matrix, we prove that the stochastic system has a unique stationary distribution, which is regarded as a stationary distribution version of small-gain theorem. These results can be applied to stochastic cooperative, competitive and predator-prey systems, or even others.
\end{abstract}



\pagestyle{myheadings}
\thispagestyle{plain}

\section{Introduction}
The present paper is concerned with the classical small-gain theorem, which was first proved by George Zames \cite{Z} in 1966.
It has been widely used as a powerful tool to investigate the robust stability of interconnected control systems. More precisely, in \cite{Z},
a sufficient condition for input-output stability of a feedback loop is that the system gain is smaller than one. It is noticed that in the original research, most of them discuss systems with linear finite-gain from input to output. For the case of nonlinear gain functions, it was first considered by Hill \cite{H} and Mareels and Hill \cite{MH}, where the notions of monotone gain were proposed and a nonlinear version of small-gain theorem was developed. Since then, more and more researchers have focused on extending the small-gain theorem to nonlinear feedback systems. In \cite{S}, Sontag introduced the concept of input-to-state stability, which was developed by Jiang, Teel and Praly in \cite{JTP} and a nonlinear ISS-type small-gain theorem was obtained. Inspired by the excellent works \cite{JTP,S}, many other nonlinear small-gain theorems have been extensively and intensively studied, see \cite{AA,ALS,DRW,ES,JM,JMW,LN,MH,NT,S,SI,ST}, which led to new applications in the design and analysis of nonlinear control systems \cite{KD,KKK} and various systems in mathematical biology for the robust stability of the feedback connection.

So far there has been a well-developed theory of the feedback analysis for deterministic systems. However, several of small-gain theorems mentioned above are highly undesirable in the application of interconnections and stabilization for more realistic models with noise disturbances, which may arise from its surrounding environmental perturbations, measurement errors, or intrinsic uncertainties of a coupling system due to high complexity. Due to the limitations of the deterministic control theory, it is natural to investigate the stochastic nonlinear control, which has been a research topic in recent years, see \cite{DK1,DK2,DK3,PB}. Recently, Marcondes de Freitas and Sontag \cite{FS1,FS2,FS3} have initiated on the study of random dynamical systems with inputs and outputs whose parameters are perturbed by so-called {\it real noise}, which generates a metric dynamical system  if it is stationary and helps one to solve random differential equations pathwisely (see \cite[p.57]{A}). Their approach  is divided into two steps: the first is for a given $``stationary"$ input $u$, to define the so-called {\it ``input-to-state characteristic operator"}  $\mathcal{K}(u)$; and the second is for a given output function, to define $``gain\ operator"$  $\mathcal{K}^{h}(u)$ by the composition of the output function and the input-to-state characteristic operator $\mathcal{K}(u)$ on the space of stationary inputs. Then they transformed the problem on obtaining a small-gain theorem into the existence of a fixed point for the gain operator. If it does exist in a manner of the Banach fixed point theorem, then the image for input-to-state characteristic operator $\mathcal{K}(\cdot)$ at the fixed point is a random equilibrium, which is globally asymptotically stable for  interconnected random systems. Such a small-gain theorem can be applied random competitive systems, which is new as far as the authors know.

It is well-known that control of stochastic differential equations is a classical field, which cannot be pathwise differential equations (see \cite[p.68]{A}). Motivated by the work of Marcondes de Freitas and Sontag \cite{FS3}, the main purpose of this paper is to carry out Marcondes de Freitas and Sontag's idea developed in random dynamical systems with inputs and outputs accompanying real noise parameters in stochastic control of interconnected control systems so that we can provide small-gain results for nonlinear stochastic systems driven by additive white noise in both trajectories and stationary distribution.  Firstly, we consider linear stochastic system driven by additive white noise with inputs and define the {\it``input-to-state characteristic operator"} $\mathcal{K}(u)$ of the system in a suitably chosen input space via backward It\^{o} integral. For a given output function $h$, we then define the $``gain\ operator"$ as the composition of output function $h$ and the input-to-state characteristic operator $\mathcal{K}(u)$ on the input space. Suppose that the output function is either order-preserving or anti-order-preserving in the usual vector order and the global Lipschitz constant of the output function is less than the absolute of the negative principal eigenvalue of linear matrix. Then we prove the so-called {\it``small-gain theorem"}: the gain operator has a unique fixed point $u$  by the Banach fixed point theorem and  input-to-state characteristic operator $\mathcal{K}(u)$ at the fixed point $u$ is a unique, globally attracting stochastic equilibrium for the random dynamical system generated by the stochastic system. Under the same assumption for the relation between the Lipschitz constant of the output function and maximal real part of stable linear matrix, we prove that the stochastic system has a unique stationary distribution by the Khasminskii theorem \cite{KHAS}, which may be regarded as a stationary distribution version of small-gain theorem.  Our result can be applied to stochastic cooperative, competitive and predator-prey systems or even others with additive white noise to obtain a globally stable stochastic equilibrium, which, illustrated in Examples 5.1-5.3, are new to us.  What we will process is much more concise than Marcondes de Freitas and Sontag's in \cite{FS3}.

The rest of this paper is organized as follows. In Section 2, we will formulate the discussed problem, review some preliminary concepts and definitions, introduce some notations and known results needed in the subsequent content, and define the input-to-state characteristic operator. Section 3 describes the asymptotic behaviour of stochastic solutions, establishes some auxiliary lemmas and presents the definition of gain operator and its properties. In Section 4, a stochastic small-gain theorem is proved and the global convergence to a unique stochastic equilibrium is presented. Section 5 presents three examples to utilize our results to stochastic cooperative, competitive and predator-prey systems, respectively. In Section 6, we summarize the small-gain theorem in trajectories, discuss to prove a small-gain theorem in stationary solution or measure.

\section{Formulation of the problem and preliminaries}
We start with a biochemical model, which contains three chemical species $X_1$, $X_2$ and $X_3$ interacted with each other as shown in Figure 2.1. In the study of molecular biology, biochemical reaction
$$
\xymatrix{& {X_1} \ar[dr]^{}& \\
     {X_3} \ar[ur]^{}&   & {X_2} \ar[ll]}$$
\vspace{-2em}
\begin{figure}[ht]
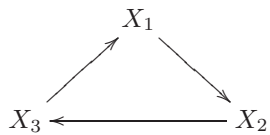

\caption{Biochemical circuit. The symbol $``X_i\rightarrow X_j"$ means that species $X_i$ represses the production of species $X_j$, $i,j=1,2,3.$}
\end{figure}

\noindent
systems are usually modeled by ordinary or partial differential equations. Furthermore, we notice that the nonlinear dynamics of biochemical reaction networks on the order of a cell is stochastic. Let us consider a nonlinear stochastic control system as shown in Figure 2.1, which consists of three elements in a feedback loop as follows:
\[dx_i=[a_ix_i+h_i(x_{i-1})]dt+\sigma_idW_t^i,\qquad i=1,2,3,\]
where $W_t(\omega)=\left(W_t^1(\omega),W_t^2(\omega),W_t^3(\omega)\right)$ is a three dimensional standard Brownian motion with $W_0^i(\omega)=0,\ i=1,2,3$, $\omega\in\Omega$. Here and below, the indices of $x$ are  taken modulo 3. This type of system can be modeled in the study of finance, statistics, engineering and multi-agent systems, where $a_1$, $a_2$ and $a_3$ are negative constants and $h_1$, $h_2$ and $h_3$ are decreasing functions of feedback.

\[\xrightarrow{\;u_1\;}X_1\xrightarrow{h_1(x_1)}\qquad \xrightarrow{\;u_2\;}X_2\xrightarrow{h_2(x_2)}\qquad \xrightarrow{\;u_3\;}X_3\xrightarrow{h_3(x_3)}\]
\vspace{-2em}
\begin{figure}[ht]
\caption{Decomposition of the feedback loop into input-output modules. In each partition, $u_i$ indicates the input
 into the element $X_i$ and $h_i(x_i)$ indicates the subsequent output---feedback of the current state.}
\end{figure}

Motivated by \cite{FS3}, we can open up the closed loop, rewriting the model as a stochastic system with inputs
\[dx_i=[a_ix_i+u_i]dt+\sigma_idW_t^i,\qquad i=1,2,3,\]
together with outputs
\[u_i(t)=h_i(x_{i-1}(t)),\qquad i=1,2,3.\]

\noindent
In view of the above analysis, we can rewrite the form of stochastic solutions by the variation-of-constants formula of stochastic differential equations, which is very important in the present work.

In general, we consider the following nonlinear stochastic system in this paper
\begin{equation}
dX_t=[AX_t+h(X_t)]dt+\sigma dW_t,\label{problem}
\end{equation}
where $W_t(\omega)=\left(W_t^1(\omega),\ldots,W_t^m(\omega)\right)$ is a two-sided time Wiener process
with values in $\mathbb{R}^m$ on the
canonical Wiener space $(\Omega,\mathscr{F},\mathbb{P})$, i.e., $\mathscr{F}$ is Borel $\sigma$-algebra of
$\Omega=C_0(\mathbb{R},\mathbb{R}^m)=\{\omega(t)\ \mbox{continuous},\ \omega(0)=0,\ t\in\mathbb{R}\}$, and let $\mathbb{P}$ be the
Wiener measure, $A=(a_{ij})_{d\times d}$ is a $d\times d$-dimensional matrix, $h:\mathbb{R}^d\rightarrow\mathbb{R}^d_+$
and $\sigma=(\sigma_{ij})_{d\times m}$ is a $d\times m$-dimensional matrix.

Let us first present the assumption for linear matrix $A$. Consider the corresponding linear ordinary differential equations:
\begin{equation}
dX_t=AX_tdt.
\label{eq1}
\end{equation}
Denote by $\Phi_j(t)=(\Phi_{1j}(t),\ldots,\Phi_{dj}(t))^T$ the solution of equations (\ref{eq1}) with initial
value $x(0)=e_j$, $j=1,\ldots,d$. Define the $d\times d$ matrix
$$\Phi(t)=(\Phi_1(t),\ldots,\Phi_d(t))=(\Phi_{ij}(t))_{d\times d}.$$
Then $\Phi(t)$ is the fundamental matrix of equations (\ref{eq1}) and $\Phi(t+s)=\Phi(t)\circ\Phi(s)$, $\forall t, s\geq0$. Assume now that $A$ is stable in the sense that all real parts of its eigenvalues are negative:
\begin{equation}\label{Astable}
{\rm Re}\ \mu \leq\lambda<0\quad{\rm for\ all\ eigenvalues}\ \mu \ {\rm of} \ A.
\end{equation}
It follows from \cite[Chap.2, Proposition 2.10]{Palis} that there is a basis of $\mathbb{R}^d$ such that the norm $\mid\cdot\mid$ satisfies that
$$<Ax,x>\ \leq \ \lambda\mid x\mid^2\quad{\rm for\ all}\ x\in \mathbb{R}^d\ {\rm and}$$
$$\|\Phi(t)\|\leq e^{\lambda t},\ t\geq0.$$

Now we propose the assumptions on $A$ as follows.

\begin{enumerate}[({A})]
\item
$A$ is {\it cooperative} in the sense that
$a_{ij}\geq0$ for all $i,j\in\{1,\ldots,d\}$ and $i\neq j$, which is stable such that (\ref{Astable}) and
\begin{equation}\label{Fun}
\|\Phi(t)\|:=\max\{|\Phi_{ij}(t)|:i,j=1,\ldots,d\}\leq{\rm e}^{\lambda t},\ t\geq0
\end{equation}
 hold.
\end{enumerate}
In this paper, we use the norm $|x|:= \max\{|x_i|: i=1,\ldots,d\}$.

Concerning the existence and uniqueness of stochastic solutions and the stability of SDEs (\ref{problem}), we make the following hypotheses:

\begin{enumerate}[({H}$_1$)]
\item
\label{isrdec1}$h\in C^1_b(\mathbb{R}^d,\mathbb{R}^d_+)$, i.e., the function $h$ and its derivatives are both bounded, and $h$ is order-preserving, i.e.,
    \[ h(x_1)\leq_{\mathbb{R}^d_+}h(x_2),\ {\rm whenever}\ x_1\leq_{\mathbb{R}^d_+}x_2\quad \]
or anti-order-preserving, i.e.,
\[ h(x_1)\geq_{\mathbb{R}^d_+}h(x_2),\ {\rm whenever}\ x_1\leq_{\mathbb{R}^d_+}x_2\quad. \]
Here, $x\leq_{\mathbb{R}^d_+}y$ means that $y-x\in\mathbb{R}^d_+$, $\forall x, y\in\mathbb{R}^d$.
\item
\label{isrdec2}Let $L=\max\{\sup\limits_{x\in\mathbb{R}^d}|\frac{\partial h_i(x)}{\partial x_j}|, i,j=1,\ldots,d\}$ such that
$-\frac{Ld^2}{\lambda}<1$.
\end{enumerate}

\noindent In cellular neural networks, $h$ can be regarded as input-output sigmoid characteristics with identical neurons, see \cite{CWK,HJ}.
By the fundamental theory of SDEs \cite{M,O}, we can obtain the existence and uniqueness of solutions for equations (\ref{problem}) with the initial value $x(0)=x_0$ $\in\mathbb{R}^d$.

Let $\varphi(t,\omega)x=x(t,\omega,x)$ be the unique solution of equations (\ref{problem}). Then using the variation-of-constants formula \cite[Theorem 3.1]{M}, we have
\begin{eqnarray}\label{eq2}
\varphi(t,\omega)x &=& \Phi(t)x+\Phi(t)\int_0^t\Phi^{-1}(s)h(\varphi(s,\omega)x)ds+\Phi(t)\int_0^t\Phi^{-1}(s)\sigma dW_s \nonumber\\
&=& \Phi(t)x+\int_0^t\Phi(t-s)h(\varphi(s,\omega)x)ds+\int_0^t\Phi(t-s)\sigma dW_s,\quad \mathbb{P}\mbox{-a.s.},\ t\geq0.
\end{eqnarray}
Before stating our main results, we will introduce some basic concepts and notations related to the theory of random dynamical systems, which can be found in \cite{A,C} for more details.

Let $(\Omega,
\mathscr{F},\mathbb{P})$ be a probability space and $(X,\varrho)$ be a separable complete metric space equipped with the Borel $\sigma$-algebra $\mathscr{B}(X)$.

\begin{definition}
A family of transformations $\{\theta_t:\Omega\mapsto\Omega,t\in\mathbb{R}\}$ with a probability space $(\Omega,
\mathscr{F},\mathbb{P})$ is
called a metric dynamical system if
\begin{remunerate}
\item it is one-parameter group, i.e.,
\[\theta_0=id,\quad \theta_t\circ\theta_s=\theta_{t+s}\quad for\ all\ t,s\in\mathbb{R};\]
\item $(t,\omega)\mapsto\theta_t\omega$ is $(\mathscr{B}(\mathbb{R})\otimes\mathscr{F}, \mathscr{F})$-measurable;
\item $\theta_t\mathbb{P}=\mathbb{P}$ for all $t\in\mathbb{R}$, i.e.,
$\mathbb{P}(\theta_tB)=\mathbb{P}(B)$ for all $B\in\mathscr{F}$ and
all $t\in\mathbb{R}$.
\end{remunerate}
\end{definition}

\begin{definition}
A random dynamical system (RDS) with the state space $X$ over a metric dynamical
system $\theta\equiv(\Omega,\mathscr{F},\mathbb{P},\{\theta_t,t\in\mathbb{R}\})$ is a $(\mathscr{B}(\mathbb{R}_+)\otimes\mathscr{F}\otimes\mathscr{B}(X),
\mathscr{B}(X))$-measurable mapping
\[\varphi:\mathbb{R}_+\times\Omega\times X\mapsto X, \quad (t,\omega,x)\mapsto\varphi(t,\omega,x),\]
such that
\begin{remunerate}
\item$\varphi(t,\omega,\cdot): X \to X$ is continuous for all $t\in
\mathbb{R}_+$ and $\omega\in\Omega$;
\item the mappings $\varphi(t,\omega):=\varphi(t,\omega,\cdot)$ form a cocycle over $\theta$
\[\varphi(0,\omega)=id,\quad \varphi(t+s,\omega)=\varphi(t,\theta_s\omega)\circ\varphi(s,\omega)\]
for all $t,s\in\mathbb{R}_+$ and $\omega\in\Omega$. Here
$\circ$ means composition of mappings.
\end{remunerate}
\end{definition}

\begin{definition}
A set-valued mapping $D:\Omega\to
2^X\backslash\{\varnothing\}$ is called a random set if for every $x\in X$, the mapping
$\omega\to {\rm dist}_X(x,D(\omega))$ is measurable, where ${\rm dist}_X(x,B)$ is the
distance in $X$ between the element $x$ and the set $B\subset X$. A random set $D$ is called a random closed (resp. compact) set if $D(\omega)$ is closed (resp. compact) in $X$ for each $\omega\in\Omega$. A random set $D$ is said to be bounded if there exist $x_0\in X$ and a random variable $r(\omega)>0$ such that
\[D(\omega)\subset\{x\in X: \varrho(x,x_0)\leq r(\omega)\}\ for\ all\ \omega\in\Omega.\]
\end{definition}

Combining the theory of stochastic differential equations and random dynamical systems, we can see that the solution of equations (\ref{problem}) generates an RDS $(\theta,\varphi)$  in $\mathbb{R}^d$, see \cite[Chap. 2]{A,C}, where $\theta$ is connected
with the Wiener process, i.e., $\theta_t\omega(\cdot):=\omega(t+\cdot)-\omega(t)$. Let $B^+:=\{B_t^+|B_t^+=W_t:t\geq0\}$ and $B^-:=\{B_t^-|B_t^-=W_{-t}:t\geq0\}$. Then $B^+$ and $B^-$ are two independent Brownian motions with one-sided time $\mathbb{R}_+$, which can be used to define the forward and backward It\^{o} integral (see Arnold \cite[p.97]{A}) as follows.

For a continuous adapted process $f$ of finite variation and $t>0$, the {\it forward It\^{o} integral} is defined by
$$\int_0^tf(s)d^+W_s:=\lim_{\Delta\rightarrow 0}{\rm in\ Pr.}\sum_{k=0}^{n-1}f(t_k)(B^+_{t_{k+1}}-B^+_{t_{k}})$$
for any partition $\Delta=\{0=t_0<t_1<\cdots t_n=t\}$; and for $t<0$, the {\it backward It\^{o} integral} is defined by
$$\int_t^0f(s)d^-W_s:=\lim_{\Delta\rightarrow 0}{\rm in\ Pr.}\sum_{k=0}^{n-1}f(t_{k+1})(B^-_{-t_{k+1}}-B^-_{-t_{k}})$$
for any partition $\Delta=\{t=t_0<t_1<\cdots t_n=0\}$.

By definitions, we can prove that
\begin{equation}\label{formula1}
\int_{t}^0f(s)d^-W_s=-\int_0^{-t}f(-s)dB^-_s\quad{\rm for\ any}\ t<0.
\end{equation}
\begin{equation}\label{formula2}
\int_{-t}^0f(-s)d^-W_s=\int_0^{t}f(t-s)d^+W_s(\theta_{-t}\omega)\quad{\rm for\ any}\ t>0.
\end{equation}

 Now we are in position to introduce the concept of {\it input-to-state characteristic  operator}. By the definition of $\theta$, we can obtain the pull-back trajectories of solutions of (\ref{problem}) as follows:
\begin{equation}
\varphi(t,\theta_{-t}\omega)x=\Phi(t)x+\int_{-t}^0\Phi(-s)h(\varphi(t+s,\theta_{-t}\omega)x)ds
+\int_{-t}^0\Phi(-s)\sigma dW_s,\ \mathbb{P}\mbox{-a.s.},\ x\in\mathbb{R}^d,
\label{eq3}
\end{equation}
where we have used (\ref{formula2}) for stochastic integral term. Returning the output function term to input function, we define the {\it input-to-state characteristic  operator}  $\mathcal{K}$:
\begin{equation}
[\mathcal{K}(u)](\omega)=\int_{-\infty}^0\Phi(-s)u(\theta_s\omega)ds+\int_{-\infty}^0\Phi(-s)\sigma dW_s,\quad\omega\in\Omega,
\label{eq4}
\end{equation}
where $u$ is a tempered random variable with respect to $\theta$, i.e.,
\[\sup_{t\in\mathbb{R}}\left\{e^{-\gamma|t|}\left|u(\theta_t\omega)\right|_2\right\}<\infty \quad {\rm for\ all}\ \omega\in\Omega\ {\rm and}\ \gamma>0,\]
where $|x|_2:=(\sum\limits_{i=1}^d|x_i|^2)^{\frac12}, x\in\mathbb{R}^d.$ In what follows, we denote $\|\Phi\|_2:=(\sum\limits_{i,j=1}^d|\Phi_{ij}|^2)^{\frac12}$, where $\Phi$ is a $d\times d$ dimensional matrix.

\textsc{{\it Remark}} 1. It is noticed that the operator $\mathcal{K}$ is well defined. In fact, for any tempered random variable $u$, since $\|\Phi(t)\|\leq e^{\lambda t}$, $t\geq0$, we have $\|\Phi(t)\|_2\leq d\|\Phi(t)\|\leq d {\rm e}^{\lambda t}$ and so
\begin{eqnarray}\left|\int_{-\infty}^0\Phi(-s)u(\theta_s\omega)ds\right|_2
&\leq&\int_{-\infty}^0\left|\Phi(-s)u(\theta_s\omega)\right|_2ds
\leq d\int_{-\infty}^0{\rm e}^{\lambda|s|}|u(\theta_s\omega)|_2ds\nonumber\\
&\leq&d\sup_{t\in\mathbb{R}}\left\{{\rm e}^{\frac\lambda2|t|}\left|u(\theta_t\omega)\right|_2\right\}
\int_{-\infty}^0e^{\frac\lambda2|s|}ds<\infty, \ \omega\in\Omega,
\nonumber\end{eqnarray}
which implies that $\lim\limits_{t\rightarrow\infty}\int_{-t}^0\Phi(-s)u(\theta_s\omega)ds$ exists for all $\omega\in\Omega$.
For any $t_1>t_2>0$,
\begin{eqnarray}\mathbb{E}\left|\int_{-t_1}^0\Phi(-s)\sigma dW_s-\int_{-t_2}^0\Phi(-s)\sigma dW_s\right|_2^2
&=&\mathbb{E}\int_{-t_1}^{-t_2}\left\|\Phi(-s)\sigma\right\|_2^2ds\nonumber\\
&\leq&\sum_{i=1}^{d}\sum_{j=1}^{m}\sigma_{ij}^2\int_{-t_1}^{-t_2}\|\Phi(-s)\|_2^2ds
\nonumber\end{eqnarray}
which together with {\rm (A)} shows that $\int_{-t}^0\Phi(-s)\sigma dW_s$ converges in $L^2$, as $t\rightarrow\infty$.  By {\rm (\ref{formula1})}, $\int_{-t}^0\Phi(-s)\sigma dW_s$ is a continuous martingale. Hence, it follows from  {\rm \cite[Problem 3.20 in Chap. 1]{KS}} that $\int_{-t}^0\Phi(-s)\sigma dW_s$
converges $\mathbb{P}\mbox{-a.s.}$ to an integrable random variable $X_\infty:=\int_{-\infty}^0\Phi(-s)\sigma dW_s$, as $t\rightarrow\infty$.
Furthermore, by the boundedness of $h$, (see (H$_1$)), we obtain that
$\{\varphi(t,\theta_{-t}\omega)x: t\geq0\}$ is a bounded set for $\mathbb{P}\mbox{-a.s.}$ $\omega\in\Omega$ and $x\in\mathbb{R}^d$.

\section{Asymptotic behaviour of RDS generated by It\^{o} SDEs}
In this section, we will give some preliminary propositions and lemmas to describe the dynamics of pull-back trajectory which will be used in the proof of our main result. To make the paper self-contained, we begin with a known result in \cite{FS3} which provides convenience for reading.

\begin{lemma}[{\rm \cite[Lemma A.2]{FS3}}]\label{lem1}
Suppose that $(x_\alpha)_{\alpha\in A}$ is a net in a normed space $X$, partially
ordered by a solid, normal cone $X_+\subseteq X$. Suppose, in addition, that the net converges
to an element $x_\infty\in X$, and that the infima and suprema
\[x_\alpha^-:=\inf\{x_{\alpha'}: \alpha'\geq\alpha\}\quad and\quad x_\alpha^+:=\sup\{x_{\alpha'}: \alpha'\geq\alpha\}\]
exist for every $\alpha\in A$. Then the nets $(x_\alpha^-)_{\alpha\in A}$ and $(x_\alpha^+)_{\alpha\in A}$ so defined also converge
to $x_\infty$.
\end{lemma}

\begin{proposition}\label{pro1}
For each $\tau>0$, let
\[a_\tau^h(\omega)=\inf\overline{\{h(\varphi(t,\theta_{-t}\omega)x):t\geq\tau\}}\]
and
\[b_\tau^h(\omega)=\sup\overline{\{h(\varphi(t,\theta_{-t}\omega)x):t\geq\tau\}},\quad x\in\mathbb{R}^d,\ \omega\in\Omega,\]
where {\rm inf} and {\rm sup} mean the greatest lower bound and the least upper bound, respectively.
Then $a_\tau^h(\omega)$ and $b_\tau^h(\omega)$ are random variables with respect to the $\sigma$-algebra $\mathscr{F}$. When $h={\rm id}$, we use the notations $a_\tau^{\rm id}$
and $b_\tau^{\rm id}$.
\end{proposition}
\begin{proof}
First, we show that $a_\tau^h(\omega)$ and $b_\tau^h(\omega)$ are well defined.
By Remark 1, it is clear that $D_\tau^h(\omega):=\overline{\{h(\varphi(t,\theta_{-t}\omega)x):t\geq\tau\}}$ and $D_\tau^{\rm id}(\omega):=\overline{\{\varphi(t,\theta_{-t}\omega)x:t\geq\tau\}}$ are two bounded sets for fixed $\tau\geq0$, $\omega\in\Omega$ and $x\in\mathbb{R}^d$, which implies that $D_\tau^h(\omega)$ and $D_\tau^{\rm id}(\omega)$ are order-bounded. Since $\mathbb{R}^d_+$ is strongly minihedral \cite[Definition 3.1.7]{C}, $a_\tau^h(\omega)$ and $b_\tau^h(\omega)$ exist.
Let $\gamma_x^\tau(\omega):=\bigcup\limits_{t\geq\tau}\{h(\varphi(t,\theta_{-t}\omega)x)\}$ be the tail of the pull-back trajectory emanating from $x$. It is noticed from Remark 1.5.1 in \cite{C} that
\[(t,x)\mapsto \varphi(t,\theta_{-t}\omega)x\ \mbox{is a continuous mapping}\]
from $\mathbb{R}_+\times\mathbb{R}^d$ into $\mathbb{R}^d$, then by Proposition 1.3.5 in \cite{C}, we have
\[\omega\mapsto\overline{\gamma_x^\tau(\omega)}:=\overline{\bigcup\limits_{t\geq\tau}\{h(\varphi(t,\theta_{-t}\omega)x)\}}\]is a random compact set with respect to $\mathscr{F}$. Together with Theorem 3.2.1 in \cite{C}, we can conclude that $a_\tau^h(\omega)$ and $b_\tau^h(\omega)$ are $\mathscr{F}$-measurable random variables in $\mathbb{R}^d_+$. When $h={\rm id}$, we can obtain the same result from Remark 1. The proof is complete.
\qquad\end{proof}

\begin{lemma}\label{lem2}
Assume that conditions {\rm (H$_1$)} and {\rm (A)} hold. Let $\varphi(t,\omega)x$ be a solution of stochastic system {\rm (\ref{problem})} with initial value $x\in\mathbb{R}^d$. Then we have
\begin{equation}\label{eq5}\mathcal{K}(\theta-\underline{\lim}\;h(\varphi))\leq\theta-\underline{\lim}\;\varphi\leq \theta-\overline{\lim}\;\varphi\leq\mathcal{K}(\theta-\overline{\lim}\;h(\varphi))\quad \mathbb{P}\mbox{-a.s.}
\end{equation}
where
\[[\theta-\underline{\lim}\;\varphi](\omega):=\lim_{\tau\rightarrow\infty}\inf\{\varphi(t,\theta_{-t}\omega)x:t\geq\tau\},\quad x\in\mathbb{R}^d,\ \omega\in\Omega\]
and
\[[\theta-\overline{\lim}\;\varphi](\omega):=\lim_{\tau\rightarrow\infty}\sup\{\varphi(t,\theta_{-t}\omega)x:t\geq\tau\},\quad x\in\mathbb{R}^d,\ \omega\in\Omega.\]
Analogously, we can define $\theta-\underline{\lim}\;h(\varphi)$ and $\theta-\overline{\lim}\;h(\varphi)$.
\end{lemma}
\begin{proof}
Here, we only prove the first inequality for the sake of convenience and the rest inequalities can be proved analogously. First, in view of the fact that \[\inf\{\varphi(t,\theta_{-t}\omega)x:t\geq\tau\}=\inf\overline{\{\varphi(t,\theta_{-t}\omega)x:t\geq\tau\}}\]
and
\[\inf\{h(\varphi(t,\theta_{-t}\omega)x):t\geq\tau\}=\inf\overline{\{h(\varphi(t,\theta_{-t}\omega)x):t\geq\tau\}}\]
by Lemma A.1 in \cite{FS3}, similar to Proposition \ref{pro1}, we can easily get that $\theta-\underline{\lim}\;h(\varphi)$ and $\theta-\underline{\lim}\;\varphi$ exist, which are also two $(\mathscr{F},
\mathscr{B}(\mathbb{R}^d))$-measurable random variables. Then $\mathcal{K}(\theta-\underline{\lim}\;h(\varphi))$ is well defined and $(\mathscr{F},
\mathscr{B}(\mathbb{R}^d))$-measurable by (\ref{eq4}) and Fubini theorem.
It is noticed that
\[[\theta-\underline{\lim}\;h(\varphi)](\omega):=\lim_{\tau\rightarrow\infty}\inf\{h(\varphi(t,\theta_{-t}\omega)x):t\geq\tau\},\quad x\in\mathbb{R}^d,\ \omega\in\Omega,\]
and by Lebesgue's
dominated convergence theorem, we have
$\mathcal{K}(\theta-\underline{\lim}\;h(\varphi))=\lim\limits_{\tau\rightarrow\infty}\mathcal{K}(a_\tau)$,
where
$a_\tau(\omega):=\inf\{h(\varphi(t,\theta_{-t}\omega)x):t\geq\tau\}$.
Therefore, we can choose an increasing sequence $\{\tau_n\}_{n\in\mathbb{N}}$ such that $\tau_n\uparrow\infty$ and it is enough to prove
\[\mathcal{K}(a_{\tau_n})\leq\theta-\underline{\lim}\;\varphi,\quad\mathbb{P}\mbox{-a.s.},\ \forall n\in\mathbb{N}.\]

\noindent By the definition of $\mathcal{K}$, for fixed $\tau_n\geq0$, we have
\begin{eqnarray}\label{eq7}
[\mathcal{K}(a_{\tau_n})](\omega) &=& \int_{-\infty}^0\Phi(-s)\inf\{h(\varphi(t,\theta_{-t}\bullet)x):t\geq\tau_n\}(\theta_s\omega)ds+\int_{-\infty}^0\Phi(-s)\sigma dW(s)\nonumber\\
&=& \int_{-\infty}^0\Phi(-s)\inf\{h(\varphi(t,\theta_{-t+s}\omega)x):t\geq\tau_n\}ds+\int_{-\infty}^0\Phi(-s)\sigma dW(s) \nonumber\\
&=& \lim_{\substack{\tilde{t}\rightarrow\infty\\\tilde{t}\geq\tau_n}}\left\{\Phi(\tilde{t})x+\int_{\tau_n-\tilde{t}}^0\Phi(-s)\inf\{h(\varphi(t,\theta_{-t+s}\omega)x):t\geq\tau_n\}ds\right.\nonumber\\
&&\left.+
\int_{-\tilde{t}}^0\Phi(-s)\sigma dW(s)\right\}\nonumber\\
&=& \lim_{\substack{\tau\rightarrow\infty\\\tau\geq\tau_n}}\inf\left\{\Phi(\tilde{t})x+\int_{\tau_n-\tilde{t}}^0\Phi(-s)\inf\{h(\varphi(t,\theta_{-t+s}\omega)x):t\geq\tau_n\}ds\right.\nonumber\\
&&\left.+
\int_{-\tilde{t}}^0\Phi(-s)\sigma dW(s):\tilde{t}\geq\tau\right\}\nonumber\\
&\leq&\lim_{\substack{\tau\rightarrow\infty\\\tau\geq\tau_n}}\inf\left\{\Phi(\tilde{t})x+\int_{\tau_n-\tilde{t}}^0\Phi(-s)\{h(\varphi(\tilde{t}+s,\theta_{-\tilde{t}}\omega)x)\}ds\right.\nonumber\\
&&\left.+
\int_{-\tilde{t}}^0\Phi(-s)\sigma dW(s):\tilde{t}\geq\tau\right\}\nonumber\\
&\leq&\lim_{\tau\rightarrow\infty}\inf\left\{\Phi(\tilde{t})x+\int_{-\tilde{t}}^0\Phi(-s)\{h(\varphi(\tilde{t}+s,\theta_{-\tilde{t}}\omega)x)\}ds\right.\nonumber\\
&&\left.+
\int_{-\tilde{t}}^0\Phi(-s)\sigma dW(s):\tilde{t}\geq\tau\right\}\nonumber\\
&=&[\theta-\underline{\lim}\;\varphi](\omega),\nonumber
\end{eqnarray}
where the fourth equality has used Lemma \ref{lem1}, while the second last inequality has applied the positivity of $\Phi(t)$ and $h$.
The proof is complete.
\qquad\end{proof}

\begin{lemma}\label{lem3}
Assume that conditions {\rm (H$_1$)} and {\rm (A)} hold. Let $\varphi(t,\omega)x$ be a solution of {\rm (\ref{problem})} with initial value $x\in\mathbb{R}^d$. Then we have
\begin{enumerate}[{\rm (i)}]
\item If $h$ is order-preserving, then
\begin{equation}\label{eq11}
h(\theta-\underline{\lim}\;\varphi)\leq\theta-\underline{\lim}\;h(\varphi)\leq\theta-\overline{\lim}\;h(\varphi)\leq h(\theta-\overline{\lim}\;\varphi) \quad \mathbb{P}\mbox{-a.s.}
\end{equation}
\item If $h$ is anti-order-preserving, then
\begin{equation}\label{eq12}
h(\theta-\overline{\lim}\;\varphi)\leq\theta-\underline{\lim}\;h(\varphi)\leq\theta-\overline{\lim}\;h(\varphi)\leq h(\theta-\underline{\lim}\;\varphi) \quad \mathbb{P}\mbox{-a.s.}
\end{equation}
\end{enumerate}
\end{lemma}
\begin{proof}
Indeed, the proof of the first inequality in (\ref{eq11}) is adequate and the rest results of this lemma can be obtained analogously.
Observe that $h$ is order-preserving, fixed $\tau\geq0$ and $x\in\mathbb{R}^d$, then for $\forall t\geq\tau$, we have
\[ h(\inf\{\varphi(t,\theta_{-t}\omega)x:t\geq\tau\})\leq h(\varphi(t,\theta_{-t}\omega)x)
\]
and
\begin{equation}\label{eq13}
h(\inf\{\varphi(t,\theta_{-t}\omega)x:t\geq\tau\})\leq\inf\{h(\varphi(t,\theta_{-t}\omega)x):t\geq\tau\}.
\end{equation}
Let $\tau\rightarrow\infty$ in (\ref{eq13}). Then
\begin{eqnarray}h(\theta-\underline{\lim}\;\varphi)(\omega)
&=&
\lim_{\tau\rightarrow\infty}h(\inf\{\varphi(t,\theta_{-t}\omega)x:t\geq\tau\})\nonumber\\
&\leq&
\lim_{\tau\rightarrow\infty}\inf\{h(\varphi(t,\theta_{-t}\omega)x):t\geq\tau\}=\theta-\underline{\lim}\;h(\varphi)(\omega).\nonumber
\end{eqnarray}
The proof is complete.
\qquad\end{proof}

\begin{lemma}\label{lem4}
Assume that conditions {\rm (H$_1$)} and {\rm (A)} hold. Then for stochastic system {\rm (\ref{problem})}, let $\varphi(t,\omega)x$ be a solution of {\rm (\ref{problem})} with initial value $x\in\mathbb{R}^d$, we have
\begin{equation}\label{eq14}
\mathcal{K}(a_\tau^h)\leq\theta-\underline{\lim}\;\varphi\leq \theta-\overline{\lim}\;\varphi\leq\mathcal{K}(b_\tau^h),\quad\mathbb{P}\mbox{-a.s.},\ \tau\geq0,
\end{equation}
where $a_\tau^h(\omega)$ and $b_\tau^h(\omega)$ are defined in Proposition {\rm \ref{pro1}}. Furthermore, define $\mathcal{K}^h:=h\circ\mathcal{K}$ as a gain operator. Then we have
\begin{enumerate}[{\rm (i)}]
\item If $h$ is order-preserving, then for fixed $\tau\geq0$
\begin{equation}\label{eq15}
(\mathcal{K}^h)^k(a_\tau^h)\leq\theta-\underline{\lim}\;h(\varphi)\leq \theta-\overline{\lim}\;h(\varphi)\leq(\mathcal{K}^h)^k(b_\tau^h),\quad\mathbb{P}\mbox{-a.s.},\ k\in\mathbb{N}.\
\end{equation}
\item  If $h$ is anti-order-preserving, then for fixed $\tau\geq0$
\begin{equation}\label{eq16}
(\mathcal{K}^h)^{2k}(a_\tau^h)\leq\theta-\underline{\lim}\;h(\varphi)\leq \theta-\overline{\lim}\;h(\varphi)\leq(\mathcal{K}^h)^{2k}(b_\tau^h),\quad\mathbb{P}\mbox{-a.s.},\ k\in\mathbb{N}.
\end{equation}
\end{enumerate}
\end{lemma}
\begin{proof}
By definitions of $a_\tau^h$ and $b_\tau^h$, it is evident that
\[a_\tau^h\leq\theta-\underline{\lim}\;h(\varphi)\leq\theta-\overline{\lim}\;h(\varphi)\leq b_\tau^h,\quad \tau\geq0.\]
Observe that $\Phi$ is monotone based on the assumption (A), then $\mathcal{K}$ is also monotone with respect to $u$,
and consequently
\[\mathcal{K}(a_\tau^h)\leq\mathcal{K}(\theta-\underline{\lim}\;h(\varphi))\leq\mathcal{K}(\theta-\overline{\lim}\;h(\varphi))\leq\mathcal{K}(b_\tau^h),
\quad\mathbb{P}\mbox{-a.s.},\ \tau\geq0,\]
which together with Lemma \ref{lem2} implies
\[\mathcal{K}(a_\tau^h)\leq\theta-\underline{\lim}\;\varphi\leq \theta-\overline{\lim}\;\varphi\leq\mathcal{K}(b_\tau^h),\quad\mathbb{P}\mbox{-a.s.},\ \tau\geq0.\]
That is, (\ref{eq14}) holds.

In what follows, we claim that (\ref{eq15}) and (\ref{eq16}) hold. We consider two cases as follows:

{\bf Case (i)} If $h$ is order-preserving, then it deduces that $h$ preserves the inequalities in (\ref{eq14}):
\[\mathcal{K}^h(a_\tau^h)\leq h(\theta-\underline{\lim}\;\varphi)\leq h(\theta-\overline{\lim}\;\varphi)\leq\mathcal{K}^h(b_\tau^h),\quad\mathbb{P}\mbox{-a.s.},\ \tau\geq0.\]
As a consequence of Lemma \ref{lem3}, it follows that
\[\mathcal{K}^h(a_\tau^h)\leq\theta-\underline{\lim}\;h(\varphi)\leq \theta-\overline{\lim}\;h(\varphi)\leq\mathcal{K}^h(b_\tau^h),\quad\mathbb{P}\mbox{-a.s.},\ \tau\geq0.\]
This proves that (\ref{eq15}) is true for $k=1$.

Next we assume that for some $k\in\mathbb{N}$, we have obtained
\[(\mathcal{K}^h)^k(a_\tau^h)\leq\theta-\underline{\lim}\;h(\varphi)\leq \theta-\overline{\lim}\;h(\varphi)\leq(\mathcal{K}^h)^k(b_\tau^h),\quad\mathbb{P}\mbox{-a.s.},\ \tau\geq0.\]
From the monotonicity of $\mathcal{K}$ and $h$, Lemma \ref{lem2} and Lemma \ref{lem3}, we can have
\begin{eqnarray}
\mathcal{K}[(\mathcal{K}^h)^k(a_\tau^h)]
&\leq&\mathcal{K}(\theta-\underline{\lim}\;h(\varphi))\leq\theta-\underline{\lim}\;\varphi\nonumber\\
&\leq&\theta-\overline{\lim}\;\varphi\leq\mathcal{K}(\theta-\overline{\lim}\;h(\varphi))
\leq\mathcal{K}[(\mathcal{K}^h)^k(b_\tau^h)],\quad\mathbb{P}\mbox{-a.s.},\ \tau\geq0.
\nonumber
\end{eqnarray}
Acted by $h$ in the above inequalities, we get that
\begin{eqnarray}
(\mathcal{K}^h)^{k+1}(a_\tau^h)
&\leq&h(\theta-\underline{\lim}\;\varphi)\leq\theta-\underline{\lim}\;h(\varphi)\nonumber\\
&\leq&\theta-\overline{\lim}\;h(\varphi)\leq h(\theta-\overline{\lim}\;\varphi)\leq(\mathcal{K}^h)^{k+1}(b_\tau^h),
\quad\mathbb{P}\mbox{-a.s.},\ \tau\geq0.
\nonumber
\end{eqnarray}
Therefore, we conclude that (\ref{eq15}) holds by mathematical induction.

{\bf Case (ii)} Assume that $h$ is anti-order-preserving, similar as the case of (i), we deduce that
\[\mathcal{K}^h(b_\tau^h)\leq h(\theta-\overline{\lim}\;\varphi)\leq h(\theta-\underline{\lim}\;\varphi)\leq\mathcal{K}^h(a_\tau^h),\quad\mathbb{P}\mbox{-a.s.},\ \tau\geq0.\]
Using (\ref{eq12}) in Lemma \ref{lem3}, we have
\[\mathcal{K}^h(b_\tau^h)\leq\theta-\underline{\lim}\;h(\varphi)\leq \theta-\overline{\lim}\;h(\varphi)\leq\mathcal{K}^h(a_\tau^h),\quad\mathbb{P}\mbox{-a.s.},\ \tau\geq0.\]
Combining the monotonicity of $\mathcal{K}$ and Lemma \ref{lem2}, it shows that
\[\mathcal{K}[\mathcal{K}^h(b_\tau^h)]
\leq\theta-\underline{\lim}\;\varphi
\leq\theta-\overline{\lim}\;\varphi
\leq\mathcal{K}[\mathcal{K}^h(a_\tau^h)],\quad\mathbb{P}\mbox{-a.s.},\ \tau\geq0,\]
which together with the anti-monotonicity of $h$ and  (\ref{eq12}) in Lemma \ref{lem3} implies
\[(\mathcal{K}^h)^2(a_\tau^h)\leq\theta-\underline{\lim}\;h(\varphi)\leq \theta-\overline{\lim}\;h(\varphi)\leq(\mathcal{K}^h)^2(b_\tau^h),\quad\mathbb{P}\mbox{-a.s.},\ \tau\geq0.\]
The rest proof of (\ref{eq16}) can be obtained analogously as the case (i) by mathematical induction.
The proof is complete.
\qquad\end{proof}
\section{Main Results}
In this section, we will state our main result on the stability of nonlinear stochastic system (\ref{problem}) and present its proof. We begin with a lemma.
\begin{lemma}\label{lem5}
Assume that conditions {\rm (H$_1$), (H$_2$)} and {\rm (A)} hold. Let $\mathcal{M}_{\mathscr{F}}^b(\Omega;[0,N])$ be the space of $\mathscr{F}$-measurable functions $f:\Omega\rightarrow[0,N]$, where $N=(N_1,\ldots,N_d)$, $N_i=\sup\limits_{x\in\mathbb{R}^d}|h_i(x)|$, $i=1,\ldots,d$. We introduce a metric on $\mathcal{M}_{\mathscr{F}}^b(\Omega;[0,N])$ as follows:
\[\rho(f_1,f_2):=|f_1-f_2|_\infty=\sup\limits_{\omega\in\Omega}|f_1(\omega)-f_2(\omega)|,\quad \forall f_1, f_2\in\mathcal{M}_{\mathscr{F}}^b(\Omega;[0,N]),\]
 then $(\mathcal{M}_{\mathscr{F}}^b,\rho)$ is a complete metric space and the gain operator
$\mathcal{K}^h:=h\circ\mathcal{K}:\mathcal{M}_{\mathscr{F}}^b\rightarrow\mathcal{M}_{\mathscr{F}}^b$  is a contraction mapping, where the definition of input-to-state characteristic  operator $\mathcal{K}$ can be chosen an $\mathbb{R}^d$-valued version for all $\omega\in\Omega$.
\end{lemma}
\begin{proof}
It is clear that $(\mathcal{M}_{\mathscr{F}}^b,\rho)$ is a metric space. Now we will show that the metric space $\mathcal{M}_{\mathscr{F}}^b$ is complete with respect to $\rho$. To prove this, we can choose a Cauchy sequence $\{f_n, n\in\mathbb{N}\}$ in $(\mathcal{M}_{\mathscr{F}}^b,\rho)$, we denote a function $f$ as follows:
\[f(\omega):=\lim_{n\rightarrow\infty}f_n(w)\in [0,N]\quad \mbox{for all}\quad\omega\in\Omega,\]
which holds based on the fact that $\{f_n(\omega), n\in\mathbb{N}\}$ is a Cauchy sequence in $\mathbb{R}^d$ for fixed $\omega\in\Omega$. It is noticed that the limit of a family of $\mathscr{F}$-measurable functions is a $\mathscr{F}$-measurable function \cite{Co}. This shows that $f\in \mathcal{M}_{\mathscr{F}}^b(\Omega;[0,N])$.

In what follows, we will prove that $|f-f_n|_\infty\rightarrow0$, as $n\rightarrow\infty$. Observing that $\{f_n(\omega), n\in\mathbb{N}\}$ is a Cauchy sequence, we know that for $\forall\varepsilon>0$, there exists $N_0=N_0(\varepsilon)\in\mathbb{N}$ such that for $n,m\geq N_0$,
\[\sup_{\omega\in\Omega}|f_m(\omega)-f_n(\omega)|<\varepsilon.\]
Let $m\rightarrow\infty$, then
\[\sup_{\omega\in\Omega}|f(\omega)-f_n(\omega)|\leq\varepsilon\quad{\rm for\ all}\ n \geq N_0,\]
which implies that $|f-f_n|_\infty\rightarrow0$, as $n\rightarrow\infty$.
Thus $(\mathcal{M}_{\mathscr{F}}^b,\rho)$ is a complete metric space.

Next we claim that $\mathcal{K}^h:\mathcal{M}_{\mathscr{F}}^b\rightarrow\mathcal{M}_{\mathscr{F}}^b$ is a contraction mapping. First, we should show that $\mathcal{K}^h:\mathcal{M}_{\mathscr{F}}^b\rightarrow\mathcal{M}_{\mathscr{F}}^b$ is well defined. From (H$_1$), it follows that $h:\mathbb{R}^d\rightarrow[0,N]$. By the definition of $\mathcal{K}$, measurability of $\theta$ and Fubini theorem, it is evident that $\mathcal{K}(u)$ is a $\mathscr{F}$-measurable function, which yields $\mathcal{K}^h:\mathcal{M}_{\mathscr{F}}^b\rightarrow\mathcal{M}_{\mathscr{F}}^b$.

Finally, we prove that $\mathcal{K}^h$ is a contraction mapping. By (H$_1$) and (H$_2$), we can have
\[\sup_{x\in\mathbb{R}^d}\|Dh(x)\|\leq L,\]
where $Dh(x)$ is the Jacobian of $h$.
Let $f_1$ and $f_2$ be two elements in $(\mathcal{M}_{\mathscr{F}}^b,\rho)$. By the fact that $|\Phi x|\leq d\|\Phi\|\cdot|x|$
 for all $x\in\mathbb{R}^d$ and $\Phi\in\mathbb{R}^{d\times d}$, we get
\begin{eqnarray}
|\mathcal{K}^h(f_1)-\mathcal{K}^h(f_2)|_\infty
&=& |h[\mathcal{K}(f_1)]-h[\mathcal{K}(f_2)]|_\infty\nonumber\\
&=& |Dh[\mathcal{K}(f_2)+\mu(\mathcal{K}(f_1)-\mathcal{K}(f_2))]\cdot[\mathcal{K}(f_1)-\mathcal{K}(f_2)]|_\infty\nonumber\\
&\leq& d\sup_{x\in\mathbb{R}^d}\|Dh(x)\|\cdot|\mathcal{K}(f_1)-\mathcal{K}(f_2)|_\infty\nonumber\\
&\leq& Ld\left|\int_{-\infty}^0\Phi(-s)f_1(\theta_s\omega)ds-\int_{-\infty}^0\Phi(-s)f_2(\theta_s\omega)ds\right|_\infty\nonumber\\
&\leq& Ld^2\int_{-\infty}^0\|\Phi(-s)\|\cdot|f_1-f_2|_\infty ds\nonumber\\
&\leq& Ld^2\int_{-\infty}^0{\rm e}^{-\lambda s}ds|f_1-f_2|_\infty\nonumber\\
&=& -\frac{Ld^2}{\lambda}|f_1-f_2|_\infty,\nonumber
\end{eqnarray}
where $0<\mu<1$, $-\frac{Ld^2}{\lambda}<1$ and the second last inequality holds due to the condition (A). The proof is complete.
\qquad\end{proof}

\begin{theorem}[{\rm Small-gain theorem}]\label{thm1}
Assume that conditions {\rm (H$_1$), (H$_2$)} and {\rm (A)} hold. Then the gain operator
$\mathcal{K}^h:=h\circ\mathcal{K}:\mathcal{M}_{\mathscr{F}}^b\rightarrow\mathcal{M}_{\mathscr{F}}^b$  possesses  a  unique  nonnegative fixed point $u\in\mathcal{M}_{\mathscr{F}}^b(\Omega;[0,N])$ such that for all $x\in\mathbb{R}^d$
\begin{equation}
\label{eq17}\lim_{t\rightarrow\infty}\varphi(t,\theta_{-t}\omega)x=[\mathcal{K}(u)](\omega)\quad \mathbb{P}\mbox{-a.s.}
\end{equation}
Moreover, $\varphi(t,\omega)[\mathcal{K}(u)](\omega)=[\mathcal{K}(u)](\theta_t\omega)$ $\mathbb{P}$-a.s., $t>0$, i.e., the image $[\mathcal{K}(u)](\cdot)$ at the fixed point $u$ for input-to-state characteristic operator  is a random equilibrium.
\end{theorem}
\begin{proof}
In view of Lemma \ref{lem4}, regardless of the monotonicity or anti-monotonicity for $h$, for fixed $\tau\geq0$, we have
\begin{equation}\label{eq18}(\mathcal{K}^h)^{2k}(a_\tau^h)\leq\theta-\underline{\lim}\;h(\varphi)\leq \theta-\overline{\lim}\;h(\varphi)\leq(\mathcal{K}^h)^{2k}(b_\tau^h),\quad\mathbb{P}\mbox{-a.s.},\ k\in\mathbb{N},
\end{equation}
where $a_\tau^h$ and $b_\tau^h$ are defined in Proposition \ref{pro1}. By Proposition \ref{pro1}, $a_\tau^h$ and $b_\tau^h$ are bounded $\mathscr{F}$-measurable functions in $(\mathcal{M}_{\mathscr{F}}^b,\rho)$. Since $\mathcal{K}^h$ is a contraction mapping on the complete metric space $\mathcal{M}_{\mathscr{F}}^b$, by the Banach fixed point theorem \cite{Y}, there exists a unique nonnegative
random variable $u:\Omega\rightarrow[0,N]$ for $\mathcal{K}^h$ such that
\[[\mathcal{K}^h(u)](\omega)=u(\omega)\quad{\rm for\ all}\ \omega\in\Omega,\]
and
\begin{equation}\label{eq19}\lim_{k\rightarrow\infty}[(\mathcal{K}^h)^{2k}(a_\tau^h)](\omega)=u(\omega)=\lim_{k\rightarrow\infty}[(\mathcal{K}^h)^{2k}(b_\tau^h)](\omega)\quad {\rm for\ all}\ \omega\in\Omega.
\end{equation}
Moreover, $u$ is independent of $x\in\mathbb{R}^d$ evidently. Combining (\ref{eq18}) and (\ref{eq19}), we can have
\[[\theta-\underline{\lim}\;h(\varphi)](\omega)=[\theta-\overline{\lim}\;h(\varphi)](\omega)=u(\omega)\quad \mathbb{P}\mbox{-a.s.}\]
which together with Lemma \ref{lem2} implies that
\[[\theta-\underline{\lim}\;\varphi](\omega)=[\theta-\overline{\lim}\;\varphi](\omega)=[\mathcal{K}(u)](\omega)\quad \mathbb{P}\mbox{-a.s.}\]
In order to  prove (\ref{eq17}), it remains to show that
\begin{equation}\label{eq20}[\theta-\underline{\lim}\;\varphi](\omega)=[\theta-\overline{\lim}\;\varphi](\omega)=\lim_{t\rightarrow\infty}\varphi(t,\theta_{-t}\omega)x,\quad \mathbb{P}\mbox{-a.s.},\ x\in\mathbb{R}^d.
\end{equation}
By definitions of infimum and supremum, it is clear that
\[\inf\{\varphi(t,\theta_{-t}\omega)x:t\geq\tau\}\leq\varphi(\tau,\theta_{-\tau}\omega)x\leq\sup\{\varphi(t,\theta_{-t}\omega)x:t\geq\tau\}, \quad \mathbb{P}\mbox{-a.s.},\ x\in\mathbb{R}^d.\]
Let $\tau\rightarrow\infty$ in the above inequality, then (\ref{eq20}) holds and so (\ref{eq17}) holds.

Furthermore, by (\ref{eq17}) and the property of continuity and cocycle for $(\theta,\varphi)$, one can show that for fixed $t>0$ and $x\in\mathbb{R}^d$, then
\begin{eqnarray}
\varphi(t,\omega)[\mathcal{K}(u)](\omega)
&=&
\varphi(t,\omega)\lim_{s\rightarrow\infty}\varphi(s,\theta_{-s}\omega)x\nonumber\\
&=& \lim_{s\rightarrow\infty}\varphi(t,\omega)\circ\varphi(s,\theta_{-s}\omega)x\nonumber\\
&=& \lim_{s\rightarrow\infty}\varphi(t+s,\theta_{-s}\omega)x\nonumber\\
&=& \lim_{s\rightarrow\infty}\varphi(t+s,\theta_{-(t+s)}\circ\theta_t\omega)x\nonumber\\
&=& [\mathcal{K}(u)](\theta_t\omega)\quad \mathbb{P}\mbox{-a.s.}
\nonumber
\end{eqnarray}
The proof is complete.
\qquad\end{proof}

\textsc{{\it Remark}} 2.
Since $(t,x)\mapsto\varphi(t,\theta_{-t}\omega)x$ is a continuous mapping from $\mathbb{R}_+\times\mathbb{R}^d$ into $\mathbb{R}^d$ {\rm \cite[Remark 1.5.1, p.34]{C}}, we can have an indistinguished version of the stochastic process $\{\varphi(t,\theta_{-t}\omega)x,\ t\in\mathbb{R}_+\}$ for fixed $x\in\mathbb{R}^d$, which leads that all conclusions presented in this paper would hold for all $\omega\in\Omega$.

\textsc{{\it Remark}} 3.
It is worth pointing out that we can not obtain our main results directly from Theorem 4.4 in \cite{FS3}. To be more
precise, the perfection of the crude cocycle $\varphi(t,\omega,x,u)$ is not easy and the small-gain condition (Definition 4.2 in \cite{FS3}) meets some difficulties for stochastic systems. In fact, the key role of the small-gain condition is to guarantee the existence on the unique, globally attracting fixed point of $\mathcal{K}^h$. In this paper, motivated by this thought, we directly consider the existence and uniqueness of globally attracting fixed points by using the Banach fixed point theorem. It is worth to note that the idea established in this paper has a potential to apply to other stochastic systems with feedback interconnections, for example, to those driven by multiplicative white noise. However, specific measurability problems, which need to be overcome,  appear during realizing this idea, see \cite{JL}.

\section{Examples}

In this section, we present several examples to illustrate the use of our small-gain theorem. Note that our result can be applied to stochastic cooperative, competitive and predator-prey systems, or to even others. As far as the authors know, there has been no criterion to guarantee them to be globally stable so far. Although we can construct many higher dimensional stochastic systems, we only give three dimensional systems in the following.

{\it Example} 5.1. Consider {\it stochastic cooperative system}:
\begin{equation}\label{EX1.1}dx_i=[(Ax)_i+h_i(x_i)]dt+\sigma_idW_t^i,\qquad i=1,2,3,\end{equation}
where
\begin{equation}\label{EX1.2}
A=\left[\begin{array}{ccc} -1&1&0\\
 1&-2&0\\
0&1&-1\end{array}\right]
\end{equation}
with three eigenvalues $\lambda_1=-1, \lambda_{2,3}=\frac{-3\pm\sqrt5}{2}$
and \begin{equation}\label{EX1.3}h_i(x_i):=\frac{1}{6+g_i(x_i)},\qquad i=1,2,3,\end{equation}
 where $g_i(x_i)=\frac\pi2-\arctan x_i$ is decreasing with respect to $x_i,\ i=1,2,3.$ It is clear that (\ref{EX1.1}) is a cooperative system.
By direct calculation, we obtain
\begin{equation}\label{EX1.4}
\Phi(t)=\left[\begin{array}{ccc}
\frac{5+\sqrt5}{10}{\rm e}^{\frac{-3+\sqrt5}{2}t}+\frac{5-\sqrt5}{10}{\rm e}^{\frac{-3-\sqrt5}{2}t}&\frac{\sqrt5}{5}{\rm e}^{\frac{-3+\sqrt5}{2}t}-\frac{\sqrt5}{5}{\rm e}^{\frac{-3-\sqrt5}{2}t}&0\\
\frac{\sqrt5}{5}{\rm e}^{\frac{-3+\sqrt5}{2}t}-\frac{\sqrt5}{5}{\rm e}^{\frac{-3-\sqrt5}{2}t}&\frac{5-\sqrt5}{10}{\rm e}^{\frac{-3+\sqrt5}{2}t}+\frac{5+\sqrt5}{10}{\rm e}^{\frac{-3-\sqrt5}{2}t}&0\\
-{\rm e}^{-t}+\frac{5+\sqrt5}{10}{\rm e}^{\frac{-3+\sqrt5}{2}t}+\frac{5-\sqrt5}{10}{\rm e}^{\frac{-3-\sqrt5}{2}t}&\frac{\sqrt5}{5}{\rm e}^{\frac{-3+\sqrt5}{2}t}-\frac{\sqrt5}{5}{\rm e}^{\frac{-3-\sqrt5}{2}t}&{\rm e}^{-t}\end{array}\right].\nonumber
\end{equation}
It is not difficult to estimate that for any $t\geq0$,
\[\|\Phi(t)\|:=\max\{|\Phi_{ij}(t)|:i,j=1,2,3\}\leq{\rm e}^{\lambda_2 t},\]
that is, (\ref{Fun}) holds.
Moreover, it is easy to see $\max\limits_{1\leq i\leq 3}{\rm Re} \lambda_i=\lambda_2<0,\ L\leq\frac{1}{36}$. So
\[-\frac{9L}{\lambda_2}\leq\frac{1}{2(3-\sqrt{5})}<1.\]
By the small-gain Theorem \ref {thm1},  stochastic cooperative system (\ref{EX1.1}) possesses a unique globally asymptotically stable random equilibrium, that is, (\ref{EX1.1}) has a unique stationary solution, which is globally attractive in pull-back trajectories.

The same conclusion holds if we replace $g_i(x_{i})$ by $g_i(x_1+x_2+x_3)$.

{\it Example} 5.2.  Consider {\it stochastic competitive system}:
\begin{equation}\label{EX2.1}dx_i=[a_ix_i+h_i(x_{i-1})]dt+\sigma_idW_t^i,\qquad i=1,2,3,\end{equation}
where $a_1=-1, a_2=-2, a_3=-3$ ($x_0=x_3$) and
\begin{equation}\label{EX2.2}h_i(x_{i-1})=\frac{1}{5+{\rm th}\ x_{i-1}}:=\frac{1}{4+g_i(x_{i-1})},\qquad i=1,2,3.\end{equation}
Since $g_i(x_{i-1})=1+{\rm th}\ x_{i-1}=1+\frac{e^{x_{i-1}}-e^{-x_{i-1}}}{e^{x_{i-1}}+e^{-x_{i-1}}}$ is increasing with respect to $x_{i-1},\ i=1,2,3,$ (\ref{EX2.1}) is a stochastic competitive  biochemical circuit. Furthermore, by (\ref{EX2.1}) and
(\ref{EX2.2}),  let $\lambda=-1$ and $L\leq\frac{1}{16}$, it follows that
\[\|\Phi(t)\|:=\max\{|\Phi_{ij}(t)|:i,j=1,2,3\}={\rm e}^{-t},\ t\geq0,\]
and
\[-\frac{9L}{\lambda}\leq\frac{9}{16}<1.\]
Applying the small-gain Theorem \ref {thm1}, stochastic competitive system (\ref{EX2.1}) admits a unique globally asymptotically stable random equilibrium, which produces an ergodic stationary solution for (\ref{EX2.1}).

Note that the same result is true when $g_i(x_{i-1})$ is replaced by $g_i(x_1+x_2+x_3)$.

{\it Example} 5.3. Consider {\it stochastic predator-prey system}:
\begin{equation}\label{EX3.1}dx_i=[(Ax)_i+h_i(x_{i-1})]dt+\sigma_idW_t^i,\qquad i=1,2,3,\end{equation}
where $x_0=x_3$, $x_4=x_1$ and
\begin{equation}\label{Predator-Prey}
A=\left[\begin{array}{ccc} -1&\sqrt[3]2&0\\
 0&-2&\sqrt[3]2\\
\sqrt[3]2 &0&-4\end{array}\right]
\end{equation}
with three eigenvalues $\lambda_1=-3, \lambda_{2,3}=-2\pm\sqrt2$ and
\begin{equation}\label{EX3.2}h_i(x_{i-1})=\frac{1}{4+\frac\pi2+\arctan x_{i-1}}:=\frac{1}{4+g_i(x_{i-1})},\qquad i=1,2,3.\end{equation}
Let $f(x)=Ax+h(x), x\in\mathbb{R}^3$. Then  $\frac{\partial f_i}{\partial x_{i+1}}(x)=\sqrt[3]2>0$ and $\frac{\partial f_{i+1}}{\partial x_i}(x)=-\frac{1}{(4+\frac\pi2+\arctan x_{i})^2}\cdot\frac{1}{1+x_i^2}<0$, for $i=1,2,3$, which implies that (\ref{EX3.1}) is a stochastic predator-prey system.
Direct calculation of $\Phi$ shows that
\begin{eqnarray}\label{EX2.3}
\Phi(t)=\left[\begin{array}{ccc}
{\rm e}^{-3t}+\frac{\sqrt2}{2}{\rm e}^{(-2+\sqrt2)t}-\frac{\sqrt2}{2}{\rm e}^{-(2+\sqrt2)t}\\
-\sqrt[3]4{\rm e}^{-3t}+\frac{\sqrt[3]4-\sqrt[6]2}{2}{\rm e}^{(-2+\sqrt2)t}+\frac{\sqrt[3]4+\sqrt[6]2}{2}{\rm e}^{-(2+\sqrt2)t}\\
\sqrt[3]2{\rm e}^{-3t}+\frac{\sqrt[6]{2^5}-\sqrt[3]2}{2}{\rm e}^{(-2+\sqrt2)t}-\frac{\sqrt[6]{2^5}+\sqrt[3]2}{2}{\rm e}^{-(2+\sqrt2)t}\end{array}\right.\nonumber\\
\left.\begin{array}{ccc}
-\sqrt[3]2{\rm e}^{-3t}+\frac{\sqrt[3]2}{2}{\rm e}^{(-2+\sqrt2)t}+\frac{\sqrt[3]2}{2}{\rm e}^{-(2+\sqrt2)t}\\
2{\rm e}^{-3t}+\frac{\sqrt2-1}{2}{\rm e}^{(-2+\sqrt2)t}-\frac{\sqrt2+1}{2}{\rm e}^{-(2+\sqrt2)t}\\
-\sqrt[3]4{\rm e}^{-3t}+\frac{\sqrt[3]4-\sqrt[6]2}{2}{\rm e}^{(-2+\sqrt2)t}+\frac{\sqrt[3]4+\sqrt[6]2}{2}{\rm e}^{-(2+\sqrt2)t}\nonumber
\end{array}\right.\nonumber\\
\left.\begin{array}{ccc}
-\sqrt[3]4{\rm e}^{-3t}+\frac{\sqrt[3]4-\sqrt[6]2}{2}{\rm e}^{(-2+\sqrt2)t}+\frac{\sqrt[3]4+\sqrt[6]2}{2}{\rm e}^{-(2+\sqrt2)t}\\
2\sqrt[3]2{\rm e}^{-3t}+(\frac{3\sqrt[6]{2^5}}{4}-\sqrt[3]2){\rm e}^{(-2+\sqrt2)t}-(\frac{3\sqrt[6]{2^5}}{4}+\sqrt[3]2){\rm e}^{-(2+\sqrt2)t}\\
-2{\rm e}^{-3t}+(\frac32-\sqrt2){\rm e}^{(-2+\sqrt2)t}+(\frac32+\sqrt2){\rm e}^{-(2+\sqrt2)t}\nonumber
\end{array}\right].\\ \nonumber
\end{eqnarray}
and it is not difficult to prove that for any $t\geq0$,
\[\|\Phi(t)\|:=\max\{|\Phi_{ij}(t)|:i,j=1,2,3\}\leq{\rm e}^{(-2+\sqrt2)t}={\rm e}^{\lambda_2 t}.\]
Specifically, since $\Phi_{ij}(t)\geq0, i,j=1,2,3$ for any $t\geq0$, due to the cooperativity of $A$. It is enough to show that
$\Psi_{ij}(t):=\frac{\Phi_{ij}(t)}{{\rm e}^{(-2+\sqrt2)t}}\leq1, i,j=1,2,3$ for any $t\geq0$. More precisely, for $\Psi_{31}(t)$, we have
\[\Psi_{31}(t)=\sqrt[3]2{\rm e}^{-(1+\sqrt2)t}+\frac{\sqrt[6]{2^5}-\sqrt[3]2}{2}-\frac{\sqrt[6]{2^5}+\sqrt[3]2}{2}{\rm e}^{-2\sqrt2t},\ t\geq0,\]
and
\[\frac{{\rm d}\Psi_{31}(t)}{{\rm dt}}=-\sqrt[3]2(1+\sqrt2){\rm e}^{-(1+\sqrt2)t}+\sqrt2(\sqrt[6]{2^5}+\sqrt[3]2){\rm e}^{-2\sqrt2t},\ t\geq0,\]
which together with $\Psi^\prime_{31}(0)>0$ implies that there exists a unique local maximum point $t_0>0$ for $\Psi_{31}(t), t\geq0$. In fact, it is clear that $t_0>0$ is also a global maximum point for $\Psi_{31}(t), t\geq0$. By direct calculation, it follows that $t_0>\frac45$ and so
\[\Psi_{31}(t)\leq\Psi_{31}(t_0)\leq\sqrt[3]2{\rm e}^{-(1+\sqrt2)\times \frac45}+\frac{\sqrt[6]{2^5}-\sqrt[3]2}{2}<1,\ t\geq0.\]
Analogously, we can obtain that there exists a unique local maximum point $t_0>0$ for $\Psi_{23}(t), t\geq0$, which implies that $\Psi_{23}(t)\leq1, t\geq0$.
Moreover, for $\Psi_{11}(t)$, we can easily see that
\[\Psi_{11}(t)={\rm e}^{-(1+\sqrt2)t}+\frac{\sqrt2}{2}-\frac{\sqrt2}{2}{\rm e}^{-2\sqrt2t},\ t\geq0,\]
and
\[\frac{{\rm d}\Psi_{11}(t)}{{\rm dt}}=-(1+\sqrt2){\rm e}^{-(1+\sqrt2)t}+2{\rm e}^{-2\sqrt2t}<0,\ t\geq0,\]
which implies that $\Psi_{11}(t)\leq1, t\geq0$. By the same method, we can obtain that the rest elements of $\Psi(t)=(\Psi_{ij}(t))_{d\times d}, t\geq0$ except for $\Psi_{33}(t)$ are monotone with respect to $t$ and consequently smaller than one. Finally, we consider the element $\Psi_{33}(t)$, it is clear that
\[\Psi_{33}(t)=-2{\rm e}^{-(1+\sqrt2)t}+(\frac32-\sqrt2)+(\frac32+\sqrt2){\rm e}^{-2\sqrt2t},\ t\geq0,\]
and
\[\frac{{\rm d}\Psi_{33}(t)}{{\rm dt}}=2(1+\sqrt2){\rm e}^{-(1+\sqrt2)t}-2\sqrt2(\frac32+\sqrt2){\rm e}^{-2\sqrt2t},\ t\geq0,\]
which together with $\Psi^\prime_{33}(0)<0$ implies that there exists a unique local minimum point $t_0>0$ for $\Psi_{33}(t), t\geq0$. It is noticed that
$\Psi_{33}(0)=1$ and $\lim\limits_{t\rightarrow\infty}\Psi_{33}(t)=\frac32-\sqrt2<1$. Then we conclude that $\Psi_{33}(t)\leq1, t\geq0$.
Furthermore, we can choose $L\leq\frac{1}{16}$, $\lambda=-2+\sqrt2$ and so
\[-\frac{9L}{\lambda}\leq\frac{9}{16(2-\sqrt2)}<1.\]
Using the small-gain Theorem \ref {thm1}, there exists a unique globally attractive random equilibrium for stochastic predator-prey system (\ref{EX3.1}).

\section{Conclusion and discussion}
In this paper, we have developed Marcondes de Freitas and Sontag's approach in random differential equations with inputs and outputs to control problem of stochastic differential equations and proved a small-gain theorem of stochastic control under the assumptions that the output function is either order-preserving or anti-order-preserving in the usual vector order and the global Lipschitz constant of the output function is less than the absolute of the negative principal eigenvalue of linear matrix. For the sake of convenience, let $\sigma=\mbox{diag}(\sigma_1,\ldots,\sigma_d)$. This means that the stochastic system (\ref{problem}) has a unique globally attracting stationary solution $[\mathcal{K}(u)](\theta_t\omega)$, whose probability distribution density is the unique stationary solution of Fokker-Planck equation
\begin{equation}\label{FPEQ}
p_t= \frac{1}{2}\sum_{i=1}^d\sigma_i^2\frac{\partial^2p(x,t)}{\partial x_i^2}- {\rm div}((Ax+h(x))p),\ x\in \mathbb{R}^d,\ t > 0,\
p(x, t) \geq 0,\ \int p(x, t)dx = 1.
\end{equation}
Such a stationary solution is ergodic. This reminds us to investigate a small-gain theorem in stationary solution or stationary measure manner for the stochastic system (\ref{problem}). We outline this approach as follows.

Assume that the matrix $A$ is stable with the maximal real part $-\lambda$ of all its eigenvalues and $h$ possesses a global Lipschitz constant $L$ with $L<\lambda-\epsilon_0$ for a sufficiently small $\epsilon_0>0$. According to \cite[Chap.2, Proposition 2.10]{Palis}, without loss of generality, we may assume that $A$ has the Jordan normal form with blocks along the diagonal of the form, i.e.,
\begin{equation}
A=\left[\begin{array}{ccc} A_1& &\\
 &\ddots& \\
& & A_q\end{array}\right]\nonumber
\end{equation}
where $A_i$ is one of the following two type of matrices:
\begin{equation}
\left[\begin{array}{cccc} \lambda_i& & &\\
\varepsilon& \lambda_i& &\\ &\ddots&\ddots&\\
& &\varepsilon& \lambda_i\end{array}\right]\quad
{\rm or}\quad
\left[\begin{array}{cccc} D_i& & &\\
I_\varepsilon&D_i& & \\ &\ddots&\ddots& \\
& & I_\varepsilon& D_i\end{array}\right]
\quad i=1,\ldots,q,\nonumber
\end{equation}
where

\begin{equation}
D_i=\left[\begin{array}{cc} a_i&-b_i\\
b_i&a_i\end{array}\right]\quad
{\rm and}\quad
I_\varepsilon=\left[\begin{array}{cc} \varepsilon&0\\
0&\varepsilon\end{array}\right],\ 0<\varepsilon<\epsilon_0.
\nonumber
\end{equation}
Define the Fokker-Planck operator by

$$
LV=\frac12\sum_{i=1}^d\sigma_i^2\frac{\partial^2V}{\partial x_i^2} + <\nabla V,(Ax+h(x))>.
$$
Here $<\cdot,\cdot>$ denotes the inner product in $\mathbb{R}^d$ and $V\in C^2(\mathbb{R}^d)$. Now let
\begin{equation}\label{V}
V(x):=\frac{1}{2}(x_1^2+x_2^2+\cdots +x_d^2).
\end{equation}
Then
\begin{eqnarray}
LV
&=&
\frac12\sum_{i=1}^d\sigma_i^2+ <x,Ax>+ <x,h(x)-h(0)>+<x,h(0)>\nonumber\\
&\leq & \frac12\sum_{i=1}^d\sigma_i^2 -(\lambda-\epsilon)|x|_2^2 +L|x|_2^2 + |h(0)|_2\cdot|x|_2\nonumber\\
&=& -(\lambda-\epsilon-L)|x|_2^2 + |h(0)|_2\cdot|x|_2+ \frac12\sum_{i=1}^d\sigma_i^2,
\nonumber
\end{eqnarray}
where $0<\epsilon<\epsilon_0$, $\lambda=-\max\limits_{1\leq i\leq q}{\rm Re} \lambda_i>L+\epsilon_0$. This implies that there is an $R$ sufficiently large such that
\begin{equation}\label{NEG}
LV\leq -\frac{1}{2}(\lambda-\epsilon-L)|x|_2^2\ \ {\rm for\ all}\ |x|_2\geq R.
\end{equation}
Combining (\ref{V}), (\ref{NEG}) and the Khasminskii theorem (see \cite[Theorem 4.1, p.108]{KHAS} and \cite[p.1163]{ZHY}), we know that there is a unique stationary solution for (\ref{FPEQ}) under the condition $L<\lambda$, the stationary measure generated by this stationary solution is ergodic (see \cite[Theorem 4.2, p.110]{KHAS})  and the transition probability function weakly converges to the stationary measure as $t$ tends to infinity. This stationary solution plays the same role as $[\mathcal{K}(u)](\theta_t\omega)$ in Theorem \ref{thm1}. As $\max\limits_{1\leq i\leq d}|\sigma_i|\rightarrow 0$, the stationary measure sequence weakly converges to a Dirac measure at the unique equilibrium of the deterministic system without noise. This means that the stationary measure concentrates around the deterministic equilibrium when the noise intensity is small.

{\it Example} 6.1. Consider two-dimensional stochastic system:
\begin{equation}\label{EX4}
    \begin{array}{l}
        \displaystyle dx_1=[-ax_1-bx_2+h_1(x_1,x_2)]dt+\sigma_1dW^1_t,\\
        \noalign{\medskip}
        \displaystyle dx_2=[bx_1-ax_2+h_2(x_1,x_2)]dt+\sigma_2dW^2_t.
    \end{array}
\end{equation}
Here $a$ and $b$ are positive, and $h=(h_1,h_2)^T$ has a global Lipschitz constant $L<a$. The above discussion shows that (\ref{EX4}) admits a unique stationary solution, which is ergodic.

It is noticed that the matrix
\begin{equation}
A=\left[\begin{array}{cc} -a&-b\\
b&-a\end{array}\right]
\nonumber
\end{equation}
is not cooperative. So Theorem \ref{thm1} cannot be applied to (\ref{EX4}). This may be an advantage of this method to investigate small-gain theorem.

Comparing to the result presented in \cite{FS3}, the method used here can be also applied to the problem (random systems) considered in \cite{FS3}. It is remarkable that the boundedness of $g$ is not necessary for us, while the derivatives of $g$ need to be bounded in the present work. For example, in Example 5.1, we can let
\begin{equation}
g_i(x_i)=\left\{\begin{array}{l} \frac\pi2-x_i,\qquad\qquad x_i\leq0;\nonumber\\
\frac\pi2-\arctan x_i,\quad x_i\geq0.\nonumber\end{array}\right.
\end{equation}
It is easy to see that the conclusion of Example 5.1 still hold.

This result is only the first step in study of stochastic stability and a small-gain theorem for stochastic control problem with multiplicative noise will be investigated in the near future, see \cite{JL}.

\section*{Acknowledgement}
The authors are greatly indebted to two anonymous referees for very careful reading our original manuscript and providing lots of very inspiring and helpful comments and suggestions which led to much improvement of the earlier version of this paper.  The authors are also very grateful to the editor Professor Zhang Qing for his valuable suggestions.

\end{document}